\documentclass[12pt,reqno,fleqn]{amsart}  
\usepackage{tikz}
\usepackage{amsmath,amstext,amsthm,amssymb,amsxtra}
\usepackage[top=1.5in, bottom=1.5in, left=1.25in, right=1.25in]	{geometry}
\usepackage{lmodern}

\usepackage{graphics}

\usepackage{mathtools}
\mathtoolsset{showonlyrefs,showmanualtags}

\usepackage{hyperref} 
\hypersetup{
	colorlinks=true,       
	linkcolor=blue,          
	citecolor=magenta,        
	filecolor=magenta,      
	urlcolor=cyan           
}

\usepackage[msc-links]{amsrefs}

\newtheorem{theorem}{Theorem}[section]

\newtheorem{lemma}{Lemma}[section]
\newtheorem{corollary}{Corollary}[section]
\newtheorem{OldTheorem}{Theorem}
\theoremstyle{definition}
\newtheorem{definition}{Definition}[section]
\theoremstyle{definition}

\theoremstyle{remark}
\newtheorem{remark}{Remark}[section]

\theoremstyle{remark}

\def\head{{\rm head}}

\def\sign{{\rm sign\,}}

\def\MM^d{\ensuremath{\mathfrak M}}
\def\MM{\ensuremath{\mathcal M}}

\def\ZZ{\ensuremath{\mathbb Z}}

\def\ZN{\ensuremath{\mathbb N}}
\def\zN{\ensuremath{\mathfrak N}}

\def\ZD{\ensuremath{\cal D}}

\def\ZD{\ensuremath{\mathcal D}}

\def\ZG{{\mathcal G\,}}

\numberwithin{equation}{section}
\def\md#1#2\emd{\ifx0#1
	\begin{equation*} #2 \end{equation*}\fi  
	\ifx1#1\begin{equation}#2\end{equation}\fi   
	\ifx2#1\begin{align*}#2\end{align*}\fi   
	\ifx3#1\begin{align}#2\end{align}\fi    
	\ifx4#1\begin{gather*}#2\end{gather*}\fi  
	\ifx5#1\begin{gather}#2\end{gather}\fi   
	\ifx6#1\begin{multline*}#2\end{multline*}\fi  
	\ifx7#1\begin{multline}#2\end{multline}\fi  
	\ifx8#1\begin{multline*}\begin{split}#2\end{split}\end{multline*}\fi
	\ifx9#1\begin{multline}\begin{split}#2\end{split}\end{multline}\fi
}
\newcommand {\e }[1]{\eqref{#1}}
\newcommand {\lem }[1]{Lemma \ref{#1}}

\newcommand {\cor }[1]{Corollary \ref{#1}}

\newcommand {\trm }[1]{Theorem \ref{#1}}

\newcommand {\sect }[1]{Section \ref{#1}}

\title[] {On generalized lacunary series}
\author{Grigori A. Karagulyan}
\address{Faculty of Mathematics and Mechanics, Yerevan State
University, Alex Manoogian, 1, 0025, Yerevan, Armenia} 
\email{g.karagulyan@ysu.am}
\author{Vahe G. Karagulyan}
\address{Faculty of Mathematics and Mechanics, Yerevan State
	University, Alex Manoogian, 1, 0025, Yerevan, Armenia} 
\email{vahekar2000@gmail.com}

\thanks{The work was supported by the Science Committee of RA, in the frames of the research project 21AG‐1A045 }

	
\begin{document}
\begin{abstract}
	Given lacunary sequence of integers, $n_k$,  $n_{k+1}/n_k>\lambda>1$, we define a new sequence $\{m_k\}$ formed by all possible $l$-wise sums $\pm  n_{k_1}\pm n_{k_2}\pm \ldots\pm n_{k_l}$.
We prove if $\lambda>\lambda_l$, then any series
\begin{equation}\label{10}
	\sum_kc_ke^{im_kx},
\end{equation}
with $\sum_k|c_k|^2<\infty$ converges almost everywhere after any rearrangement of the terms, where $1<\lambda_l<2$ is a certain critical value.
We establish this property, proving a new Khintchine type inequality $\|S\|_p\le C_{l,\lambda,p}\|S\|_2$, $p>2$, where $S$ is a finite sum of form \e{10}.
 For $\lambda\ge 3$, we also establish a sharp rate $p^{l/2}$ for the growth of the constant $C_{l,\lambda,p}$ as $p\to\infty$. Such an estimate for the Rademacher chaos sums was proved independently by Bonami \cite{Bon} and Kiener \cite{Kie}.  In the case of $\lambda\ge 3$ we also establish some inverse convergence properties of series \e{10}: 1) if series \e{10} converges a.e., then $\sum_k|c_k|^2<\infty$, 2) if it a.e. converges to zero, then  $c_k=0$.
\end{abstract}

	\maketitle  
\section{Introduction}
\subsection{An historical overview} Let $\{r_n\}_{n\ge 1}$ be the sequence of Rademacher functions $r_n(x)=\sign\left(\sin 2^n\pi x\right)$ on $(0,1)$.
The classical Khintchine inequality states that for any $0<p<\infty$ there are constants $A_p,B_p$ such that
\begin{equation}\label{1}
 A_p\left\|S\right\|_2\le \left\|S\right\|_p\le B_p\left\|S\right\|_2,
\end{equation}
for every Rademacher polynomial $S$. The Khintchine inequality is a well-known object in analysis and probability with various generalizations and applications. A special case of the inequality was first studied by Khintchine \cite{Khi}, proving the right bound of \e{1} with $B_p=\sqrt{p/2+1}$ and $p\ge 2$. Further study of  the inequality were given by Littlewood \cite{Lit}, Paley and Zygmund \cite{PaZy}. 
Let $A_p$ and $B_p$ denote the best constants, for which inequality \e{1} holds. It is trivial that $A_p=1$ if $2\le p<\infty$ and $B_p=1$ for all $0<p\le 2$, while it took the work of many mathematicians to settle all the other cases. Stechkin \cite{Ste} computed $B_{p}$ for even integers $p\ge 4$, which then extended for all real numbers $p\ge 3$ by Young \cite{You}. Then solving a long-standing problem of Littlewood (see. \cite{Lit}) Szarek \cite{Sza} proved that $A_1=1/\sqrt{2}$. Haagerup in \cite{Haa} introduced a new method, computing the sharp constants in remaining cases and  recovering the prior results newly.

For a given integer $l\ge 1$ denote by $\ZD(l)$ the set of integers $m$, permitting a dyadic representation 
\begin{equation}\label{d70}
	m=2^{k_1}+2^{k_2}+\ldots+2^{k_l},
\end{equation}
and let 
\begin{equation}\label{d74}
	w_m(x)=r_{k_1}(x)\ldots r_{k_l}(x)
\end{equation}
be the corresponding Walsh function. In some literature the system $\{w_m:\, m\in \ZD(l)\}$ is called Rademacher chaos of order $l$ (\cite{Mul, Ast1}). This subsystem of the Walsh functions has been investigated from different point view: boolean functions, theory of orthogonal series, probability theory \cite{Ast1, Mul, Don}. It has specific properties that the complete system of Walsh functions doesn't.
The following well-known inequality proved independently by Bonami \cite{Bon} and Kiener \cite{Kie} (see also \cite{Don} chap. 9, \cite{Blei} chap. 7) provides a generalization of the Khintchine classical inequality for the Rademacher chaos polynomials 
\begin{equation}\label{d1}
	S(x)=\sum_{m\in \ZD(l)}a_mw_m(x),
\end{equation}
where $a_m$ are real numbers. It plays significant role in the theory of boolean functions (see \cite{Don}, chap. 9).
\begin{OldTheorem}[Bonami-Kiener]
For any integers $l\ge 2$ and any sum in \e{d1} it holds the bound
\begin{equation}\label{d2}
	\|S\|_p\le (p-1)^{l/2}\|S\|_2,\quad p>2.
\end{equation}
\end{OldTheorem} 
It was also proved in \cite{Bon} that the constant growth in \e{d2} is optimal in the sense that  for any $p>2$ there exists a sum \e{d1} such that 
$	\|S\|_p\ge C_lp^{l/2}\|S\|_2$ with a constant $C_l>0$, depending only on $l$. Using a standard argument one can deduce from \e{d2}  the bound
\begin{equation}\label{d3}
	\|S\|_p\le C_{p,q,l}\|S\|_q,
\end{equation} 
for any Rademacher chaos \e{d2}, where $1\le q<p<\infty$. In contrast to classical Rademacher case (when $l=1$) to the best of our knowledge the optimal constant that can be in \e{d3}, is not known for any combination of the parameters $p> q$. Some estimates of the optimal constant of \e{d3} one can find in papers \cite{Jan, Don, Lar, Ivan}.

The following definition is well known in the theory of orthogonal series. 
\begin{definition}
	An orthogonal system $\{\phi_n,\, n=1,2,\ldots\}\subset L^2(0,1)$ is said to be a convergence system, if the condition
\begin{equation}\label{d14}
	\sum_{n=1}^\infty a_n^2<\infty,
\end{equation}
implies almost everywhere convergence of orthogonal series
\begin{equation}\label{d13}
	\sum_{n=1}^\infty a_n\phi_n(x).
\end{equation}
If \e{d13} converges a.e. after any rearrangement of the terms, then we say $\{\phi_n\}$ is an unconditional convergence system.
\end{definition} 
It is well-known that the trigonometric and  Walsh systems are convergence systems (see \cite{Car, Bil}), but none of them is an unconditional convergence system (see \cite{Uly1, Uly2, Ole}). Moreover, Ul\cprime yanov \cite{Uly1} and Olevskii \cite{Ole} proved that unconditionality fails for any complete orthonormal system. According to a classical result of Stechkin \cite{Ste} (see also \cite{KaSa}, chap. 9.4), if an orthonormal system $\{\phi_n\}$ satisfies the Khintchine type inequality
\begin{equation}
	\left\|\sum_{k=1}^n a_k\phi_k\right\|_p\le c	\left\|\sum_{k=1}^n a_k\phi_k\right\|_2
\end{equation}
for some $p>2$, then it is an unconditional convergence system. So  from inequality \e{d2} it follows that the Rademacher chaos of any order is an unconditional convergence system.

Inverse convergence properties of orthonormal systems give characterization of the coefficients $\{a_n\}$ based on certain convergence conditions on series $\e{d13}$. Classical examples of orthonormal systems having inverse properties are Rademacher and lacunary trigonometric systems. It is well-known if the Rademacher series 
\begin{equation}\label{d19}
	\sum_{n=1}^\infty a_nr_n(x)
\end{equation}
converges on a set of positive measure, then the coefficients satisfy \e{d14} (Kolmogorov, \cite{KaSt}, chap. 4.5), and if  \e{d19} converges to zero on a set of measure $>1/2$, then $a_n=0$ for all $n$ (Stechkin-Ul\cprime yanov \cite{StUl}). 

The Rademacher chaos in the context of inverse properties was first considered by Astashkin and Sukhanov \cite{Ast2} (see also \cite{Ast1}, chap. 6). 
\begin{OldTheorem}[Astashkin-Sukhanov, \cite{Ast2}]\label{O2}
	Let $\{m_k\}$ be the increasing numeration of $\ZD(l)$. If series
	\begin{equation}\label{d18}
		\sum_{k\ge 1}a_{k}w_{m_k}(x)
	\end{equation}
	converges in measure, then 
	\begin{equation*}
		\sum_{k\ge 1}|a_k|^2<\infty.
	\end{equation*}
\end{OldTheorem}
\begin{OldTheorem}[Astashkin-Sukhanov, \cite{Ast2}]\label{O3}
If series \e{d18} converges to zero on a set $E\subset (0,1)$ of measure $|E|>1-2^{-l}$, then $a_{k}=0$ for all $k\ge 1$.
\end{OldTheorem}
It was also shown in \cite{Ast2} the sharpness of the condition $|E|>1-2^{-l}$ in \trm{O3}. More precisely, it was given an example of non-trivial series \e{d19},
which terms vanish on a set of measure $1-2^{-l}$.

\subsection{Generalized lacunary trigonometric sums}
A sequence of positive integers $\zN=\{n_k:\, k=1,2,\ldots\}$ is said to be $\lambda$-lacunary if 
\begin{equation}\label{d54}
	n_{k+1}/ n_k>\lambda>1.
\end{equation}
The trigonometric functions
\begin{equation}
\{e^{\pm in_kx}:\, k=1,2,\ldots\}
\end{equation}
corresponding to a $\lambda$-lacunary $\{n_k\}$ have many common properties with Rademacher functions. The analogue of the Khintchine inequality \e{1} was obtained by Zygmund (\cite{Zyg}, chap. 5.8) and here again the optimal  constant $B_p$ satisfy the relation $B_p\sim \sqrt p$ as $p\to \infty$.  Zygmund \cite{Zyg1, Zyg2} also proved inverse properties for lacunary trigonometric series. Namely, if the lacunary sums
\begin{equation}\label{d53}
	\sum_{k=-m}^{m} c_ke^{ in_kx},\quad m=1,2,\ldots, \quad (n_{-k}=n_k)
\end{equation} 
converge on a set $E$ of positive measure, then $\sum_k|c_k|^2<\infty$,  if sums \e{d53} converge to zero on $E$, then $c_k=0$.

To the best of our knowledge the analogue of Rademacher chaos  for lacunary trigonometric systems was not previously considered. 
\begin{definition}
	Let $\mathfrak{N}=\{n_k:\, k=1,2,\ldots \}$ be an increasing sequence of integers, $1\le n_1< n_2< \ldots $. For $l\ge 2$ denote by $\zN(l)$ the set of integers $m\in \ZZ$, having the representation
	\begin{equation}\label{d31}
		m=\varepsilon_1 n_{k_1}+\varepsilon_2n_{k_2}+\ldots+\varepsilon_{l}n_{k_l},
	\end{equation}
where $\varepsilon_j=\pm 1$ and $k_1>k_2>\ldots>k_l$, and let $\zN^*(l)=\cup_{1\le s\le l}\zN(s)$.
\end{definition} 
We will consider the $\zN(l)$ for a $\lambda$-lacunary sequences $\zN$. The corresponding generalized lacunary system
\begin{equation}\label{d62}
	\{e^{imx}:\, m\in \zN(l)\}
\end{equation}
will serve as an analogue of Rademacher chaos of order $l$. Let $\lambda_l>1$ be the single solution of the equation
\begin{equation}\label{d28}
	x^{l-1}=x^{l-2}+\ldots+1,\quad x>1,
\end{equation}
where $l\ge 2$ is an integer. 
\begin{theorem}\label{T1}
Let $l\ge 2$,  $\lambda>\lambda_l$ and $\zN$ be a $\lambda$-lacunary sequence of integers. Then for any finite sum
\begin{equation}\label{d26}
	S(x)=\sum_{m\in \zN(l)}c_me^{2\pi imx},
\end{equation}
it holds the bound
\begin{equation}\label{d29}
	\|S\|_p\le C_{l,\lambda,p}\|S\|_2, \quad p>2.
\end{equation}
\end{theorem}
\begin{corollary}\label{C1}
Let $l\ge 2$, $\lambda>\lambda_l$ and $\zN$ be a $\lambda$-lacunary sequence. Then the trigonometric functions \e{d62} form an unconditional convergence system. Namely, for any increasing sequence of integer sets $\ZG_n\subset \zN(l)$ the sums 
\begin{equation}\label{d27}
	\sum_{m\in \ZG_n}c_me^{imx},\quad n=1,2,\ldots,
\end{equation}
converge a.e. as $n\to\infty$, whenever 
\begin{equation}\label{d35}
	\sum_{m\in \zN(l)} |c_m|^2<\infty.
\end{equation}
Conversely, the $\lambda_l$ is the critical bound here. That is for any $l\ge 2$ there exists a $\lambda_l$-lacunary sequence $\zN$ and a choice of coefficients \e{d35} such that for certain increasing sequence of integer sets $\ZG_n\subset \zN(l)$ sums \e{d27} diverge a.e. . 
\end{corollary}
\begin{remark}
	Note that \cor {C1} immediately follows from \trm{T1} combined with the above mentioned result of Stechkin \cite{Ste}. So the condition $\lambda>\lambda_l$ is sharp also in \trm{T1}. Moreover, one can say that for a $\lambda_l$-lacunary sequence $\zN$ we have
	\begin{equation*}
		\sup_S\|S\|_p/\|S\|_2=\infty
	\end{equation*}
for any $p>2$. Here sup is taken over all finite nontrivial sums \e{d26}. 
\end{remark}
\begin{remark}
	Observe that $1<\lambda_l<2$ and $\lambda_l\nearrow 2$, as $l\to\infty$. Also, we have $\lambda_2=1$ and so in the case  $l=2$ results of  \trm{T1} and \cor{C1} hold for any $\lambda>1$. 
\end{remark} 
Along with $\ZD(l)$ (see \e{d70}) we consider also the enlarged family $\ZD^*(l)=\cup_{1\le s\le l}\ZD(s)$ of integers $m$, permitting a dyadic representation 
\begin{equation}\label{d61}
	m=2^{k_1}+2^{k_2}+\ldots+2^{k_s},\quad 1\le s\le l.
\end{equation}
Hence, as a particular case of \cor {C1} we can claim the following. 
\begin{corollary}
	Let $l\ge 1$ and $\{m_k\}$ be a numeration of the integers in $\ZD^*(l)$. 
Then $\{e^{im_kx}\}$ is an unconditional convergence system.
\end{corollary}

For $3$-lacunary sequences we also prove the following full version of the Bonami-Kiener inequality \e{d2} for the generalized lacunary sums.
\begin{theorem}\label{T2}
		If $l\ge 2$ and $\zN$ is a $3$-lacunary sequence, then for any finite sum \e{d26} we have
	\begin{equation}\label{d50}
 \|S\|_p\le (8(p-1))^{l/2}\|S\|_2.
	\end{equation}
\end{theorem}

We will also consider inverse properties of generalized lacunary systems with the following definition of lacunarity. 
\begin{definition}
	Let $l\ge 2$ and $\mathfrak{N}=\{n_k:\, k=1,2,\ldots \}$ be an increasing sequence of integers. Denote by $\zN_+(l)$ the set of integers $m\in \ZZ$, having a representation
	\begin{equation}\label{d58}
		m=n_{k_1}+n_{k_2}+\ldots+n_{k_l},
	\end{equation}
	where $k_1>k_2>\ldots>k_l$ and let $\zN_+^*(l)=\cup_{1\le s\le l}\zN_+(l)$.
\end{definition} 

 The next results will be stated in the setting of general matrix summation methods $T=\{t_{n,m}:\, m\in \ZZ,\, n=1,2,\ldots\}$, satisfying the relations
\begin{align}
	&1)\, \{t_{n,m}:\, m\in \ZZ\} \text { is finite for any } n=1,2,\ldots,\\
	&2)\, |t_{n,m}|\le M,\\
	&3)\, \lim_{n\to\infty}t_{n,m}=1\text{ for any } m=1,2,\ldots.\label{d60}
\end{align}
Here and everywhere below we say a numerical sequence $\{b_n\}$ is finite if it has finite number of non-zero terms. We say that a numerical series $\sum_{m\in \ZZ}a_m$ is $T$-summable to $S$ if
\begin{equation*}
	\lim_{n\to\infty}\sum_mt_{n,m}a_m=S.
\end{equation*}
 An example of such a summation method can be given by an increasing sequence of finite integer sets $\ZG_n\subset \ZZ$. Then the convergence of the sums $	\sum_{m\in \ZG_n}b_n$ is the summability, corresponding to the matrix
\begin{equation}\label{d63}
t_{n,m}	=\left\{\begin{array}{ll}
		1&\hbox{ if } m\in \ZG_n ,\\
		0&\hbox{ otherwise.  }
	\end{array}
	\right.
\end{equation}
Thus the convergence of series after some rearrangement of the terms is a case of such summability. In the next results we will always let $T=\{t_{m,n}\}$ be a summation method, satisfying 1)-3). 
\begin{theorem}\label{T6}
Let $l\ge 2$, $\zN$ be a $\lambda$-lacunary sequence with $\lambda>\lambda_{l+1}$, and suppose $T=\{t_{n,m}\}$ satisfies 1)-3). If
\begin{equation}
		\sup_n\left|\sum_{m\in \zN_+^*(l)}t_{n,m}c_me^{imx}\right|<\infty
\end{equation}
on a set $E$ of measure 
\begin{equation}\label{d57}
	|E|>\alpha(l,\lambda),
\end{equation}
then $\sum_{m\in \zN_+^*(l)}|c_m|^2<\infty$. Here $\alpha(l,\lambda)\in (0,1)$ is a constant, depending only on $l$ and $\lambda$.
\end{theorem}
\begin{theorem}\label{T7}
	Under the conditions of  \trm{T6}, if the sums
	\begin{equation}
		\sum_{m\in \zN_+^*(l)}t_{n,m}c_me^{imx}
	\end{equation}
	converge to zero on a set $E$ of measure \e{d57}, then $c_m=0$ for all $m\in \zN_+^*(l)$.
\end{theorem}
\begin{remark}
	\trm {T7} provides a uniqueness property of rare trigonometric series, which is not true for the full trigonometric series. Moreover,  Menshov \cite{Men} was the first who constructed a non-trivial trigonometric series, converging to zero almost everywhere. Such series in the theory of orthogonal series are called null-series.
\end{remark}
	A significant case of Theorems \ref{T6} and \ref{T7} is when the summation matrix coincides with \e{d63}. Thus we can state the following result, where it is important that the sequence $\{m_k\}$ is not required to be increasing.
\begin{corollary}\label{C4}
	Let $\{m_k\}$ be a numeration of integers $\ZD^*(l)$ (not necessarily increasing). Then the series 
	\begin{equation}
		\sum_{k}c_ke^{im_kx}
	\end{equation}
converges a.e. if and only if 
\begin{equation*}
	\sum_k|c_k|^2<\infty.
\end{equation*}
\end{corollary}
	In fact, the proofs of both Theorems \ref{T6} and \ref{T7} immediately follows from \lem{L11} that is an inverse Parseval inequality. Proving a similar lemma, such results we establish for the Walsh functions too.
\begin{theorem}\label{T4}
 Let $\{m_k\}$ be a sequence of integers $\ZD^*(l)$ and a summation $T=\{t_{m,n}\}$ satisfy 1)-3). If  Walsh sums satisfy
	\begin{equation}\label{d10}
		\sup_n\left|\sum_{k=1}^\infty t_{n,k}a_kw_{m_k}(x)\right|<\infty
	\end{equation}
on a set $E\subset (0,1)$ of measure
	\begin{equation}\label{d11}
		|E|>1-2^{-4l},
	\end{equation}	
	then it follows that $\sum_ka_k^2<\infty$.
\end{theorem}
\begin{theorem}\label{T5}
Under the hypothesis of \trm{T4}, if the Walsh sums 
\begin{equation}\label{d72}
		\sum_k t_{n,k}a_kw_{m_k}(x)
\end{equation}
converge to zero on a set $E\subset (0,1)$ of measure \e{d11} as $n\to \infty$, then $a_k=0$ for all $k=1,2,\ldots$.
\end{theorem}

Note that both Theorems \ref{T4} and \ref{T5} use the sequence $\ZD^*(l)$. Taking $\ZD(l)$ in \trm{T5}, the bound \e{d11} may be improved. 
\begin{theorem}\label{T8}
 Let $\{m_k\}$ be a sequence of integers $\ZD(l)$ and a summation $T=\{t_{m,n}\}$ satisfy 1)-3). If the Walsh sums 
\e{d72} converge to zero on a set $E\subset (0,1)$ of measure 
\begin{equation}\label{d73}
	|E|>1-2^{-l},
\end{equation}	
 then $a_k=0$ for all $k=1,2,\ldots$.
\end{theorem}
\subsection{Final remarks}
\begin{remark}
 \trm{T4} is gives a generalization of \trm{O2}. First,  \trm{T4} is stated in the setting of general summation method and second, the convergence in measure is replaced to the boundedness of sums \e{d72} on a set of measure \e{d11}. The method of the proof of \trm{T4}  is quite different from the one used in \trm{O2}, where authors apply the decoupling technique that is applicable only for the series with the canonical order of the Walsh functions. In fact, our proof of \trm{T4} (as well us \trm{T5}) is elementary and immediately follows from an inverse Parseval type inequality proved in \sect{S6} (\lem{L13}). 
\end{remark}
\begin{remark}
	In the proof of \trm{T8} we use the argument of Astashkin-Sukhanov \cite{Ast2}. The main ingredient here are Lemmas \ref{L10} and \ref{L12} proved in \sect {S10}. 
\end{remark}
\begin{remark}
In fact, \trm{O3} of \cite{Ast2} was proved in more general setting, instead of the Walsh functions \e{d74} considering the functions
\begin{equation*}
	g_m(x)=\phi\left(2^{k_1}x\right)	\ldots\phi\left(2^{k_l}x\right) (\text{ see } \e{d70}), 
\end{equation*}
where $\phi$ is a $1$-periodic finite-valued function such that
\begin{equation*}
	\phi\left(x+\frac{1}{2}\right)-\phi(x)\neq 0 \text { for a.e. }x\in (0,1). 
\end{equation*} 
The same generalization permits also \trm{T8}, it just needed to replace $\pm 2^l$ in \lem{L10} by a non-zero number. One can apply this general results, taking $\phi(x)$ to be either the Rademacher first function $r_1(x)$ or $\sin (2\pi x)$.
\end{remark}
\begin{remark}
	The Littlewood-Paley inequality is a basic tool in the proof of \trm{T1}.  Moreover, it is shown that the constants in \e{d29} and \e{d41} obey the estimate
	\begin{equation*}
		C_{l,\lambda,p}\le c(l,\lambda) \cdot C_p.
	\end{equation*}
 The optimal rate of growth of the constant in \e{d41} is an open problem yet. It is only known the bound $\sqrt{p}\lesssim C_p\lesssim p\log p$ due to Pichorides \cite{Pich}, where it is also conjectured that $C_p\lesssim \sqrt{p}$.
 Hence we conclude, in fact, \e{d29} holds with a constant $c(l,\lambda)p\log p$.
\end{remark}

We would like to thank Boris Kashin and Sergei Astashkin for valuable discussions on the subject during the 21th Saratov (Russia) Winter School on Theory of Functions and Approximations.	

\section{Proof of Theorems \ref{T1} and \ref{T2}}
The proof of \trm{T1} uses the approach given in \cite{Mul} in the case of Rademacher chaos.  The main ingredient here is the Littlewood-Paley inequality
\begin{equation}\label{d41}
	\left\|\sum_{k\ge 1}c_ke^{2\pi ikx}\right\|_p\le C_p\left\|\left( \sum_{j\ge 0}\left|\sum_{k\in [2^j,2^{j+1})}c_ke^{2\pi ikx}\right|^2\right)^{1/2}\right\|_p.
\end{equation}
The following standard properties of lacunary sequences are either well-known (see for example \cite{Zyg}, chap. 5.7, or \cite{Bari}, chap. 11) or easy to verify. For an integer $m$ from \e{d31} we denote $\head(m)=\varepsilon_{k_1}n_{k_1}$ and let 

(P1) If $l\ge 2$ and $\zN=\{n_k\}$ is a $\lambda$-lacunary sequence with $\lambda>\lambda_l$, then 
	\begin{equation}\label{d40}
		\ZG_j=\{m\in \zN^*(l):\, \head(m)=\sign(j)\cdot n_j\}\subset \pm (a_l\cdot n_j,b_l\cdot n_j),
	\end{equation}
where $b_l>a_l>0$ are constants $j\in \ZZ\setminus \{0\}$.

(P2) If $\{n_k\}$ is $3$-lacunary, then every integer $m$ may have at most one representation
\begin{equation}\label{d80}
	\varepsilon_1n_{j_1}+\ldots+\varepsilon_sn_{j_s},
\end{equation}
where  $s=1,2,\ldots,$, $\varepsilon_j=\pm 1$, $j_1>\ldots>j_s$. Besides, $0$ can not be written in the form \e{d80}.

(P3) Let $l\ge 2$ and  $\{n_k\}$ be a $\lambda$-lacunary sequence with $\lambda>\lambda_{l+1}$. Then every integer $m$ has at most one representation
\begin{equation}\label{d4}
	m=n_{j_1}+\ldots+n_{j_s}, 
\end{equation} 
where  $j_1>\ldots>j_s$, $1\le s\le l$. 

(P4) Let $l\ge 2$ and  $\{n_k\}$ be a $\lambda$-lacunary sequence with $\lambda>\lambda_{l+1}$. Then every integer $m$ has at most $d(l,\lambda)$ number of representation
	\begin{equation}\label{d43}
		m=n_{j_1}+\ldots+n_{j_s}-n_{k_1}-\ldots-n_{k_t}, 
	\end{equation} 
where  
\begin{align}
&0\le s,t\le l,\\
&\{j_1,\ldots,j_s\}\cap\{ k_1,\ldots,k_t\}=\varnothing,\\
&j_1>\ldots>j_s,\quad k_1>\ldots>k_t,
\end{align}
and $d(l,\lambda)$ is an integer depending only on $l$ and $\lambda$.

\begin{proof}[Proof of \trm{T1}]
	We will use induction. Suppose we have already proved \e{d29} for $l=s-1$. Let us proceed the case of $l=s$. Using (P1) and the Littlewood-Paley inequality \e{d41}, for any polynomial \e{d26} we can write
	\begin{equation}\label{d24}
		\|S\|_p\le C\left\|\left(\sum_j\left|\sum_{m\in \ZG_j}c_me^{2\pi imx}\right|^2\right)^{1/2}\right\|_p,
	\end{equation}
where $C=C_{l,\lambda,p}$. If $m\in \ZG_j$, $m=\varepsilon_j n_j+t_m$, where
\begin{equation*}
 \varepsilon_j =\pm 1,\quad t_m=\varepsilon_1n_{j_1}+\ldots+\varepsilon_{l-1}n_{j_{s-1}},\, j>j_1>\ldots>j_{s-1}.
\end{equation*}
Substituting it in \e{d24} then applying Minkowski's inequality in $L^{p/2}$, we obtain
\begin{align}
	\|S\|_p&\le C\left\|\left(\sum_j\left|\sum_{m\in \ZG_j}c_me^{2\pi it_mx}\right|^2\right)^{1/2}\right\|_p\\
	&\le C\left(\sum_j\left\|\left|\sum_{m\in \ZG_j}c_me^{2\pi it_mx}\right|^2\right\|_{p/2}\right)^{1/2}=C\left(\sum_j\left\|\sum_{m\in \ZG_j}c_me^{2\pi it_mx}\right\|_{p}^2\right)^{1/2}.
\end{align}
Since $\{t_m:\, m\in \ZG_j\}\subset \zN(s-1)$ and $\lambda>\lambda_s>\lambda_{s-1}$, applying the induction assumption, we obtain
\begin{equation*}
	\left\|\sum_{m\in \ZG_j}c_me^{2\pi it_mx}\right\|_{p}\le C'_{l,\lambda,p}\left(\sum_{m\in \ZG_j}|c_m|^2\right)^{1/2}
\end{equation*}
and then $\|S\|_p\le  C'_{l,\lambda,p}\|S\|_2$. This completes the proof of \trm{T1}.
\end{proof}

\begin{lemma}\label{D1}
	If $\{n_k\}$ is a $3$-lacunary sequence, then
	\begin{align}
		\int_0^1e^{2\pi im(x+u)}&\prod_{k=1}^s\cos (2\pi n_{k_j}u)du\\
		&\qquad=\left\{\begin{array}{ll}
			2^{-s}e^{2\pi imx},&\hbox{ if }m= \varepsilon_1n_{k_1}+\varepsilon_{2}n_{k_2}+\ldots+\varepsilon_1n_{k_s},\\
			0,&\hbox{ otherwise },
		\end{array}
		\right.
	\end{align}
where $\varepsilon_j=\pm 1$, $k_1>k_2>\ldots>k_s$.
\end{lemma}
\begin{proof}
	Using the product of two cosines formula repeatedly, we get the formula 
	\begin{equation}\label{d46}
		\prod_{k=1}^s\cos (2\pi n_{k_j}x)=2^{-s}\sum_{\varepsilon_i=\pm 1}\cos2\pi(\varepsilon_1n_{k_1}+\varepsilon_{2}n_{k_2}+\ldots+\varepsilon_sn_{k_s})x,
	\end{equation}
	where the summation is taken over all combinations of $\varepsilon_j=\pm 1$. Thus if $m\neq  \varepsilon_1n_{k_1}+\varepsilon_{2}n_{k_2}+\ldots+\varepsilon_1n_{k_s}$ for any choice of $\varepsilon_j=\pm 1$, then by orthogonality we obtain
	\begin{equation*}
		\int_0^1e^{2\pi im(x+u)}\prod_{k=1}^s\cos( 2\pi n_{k_j}u)du=0.
	\end{equation*}
Now let $m= \varepsilon'_1n_{k_1}+\varepsilon'_{2}n_{k_2}+\ldots+\varepsilon'_1n_{k_s}$, for some  $\varepsilon'_j=\pm 1$. By (P2), such a representation of $m$ is unique. Thus, using \e{d46}, we conclude
	\begin{align}
		\int_0^1e^{2\pi im(x+u)} &\prod_{k=1}^s\cos( 2\pi n_{k_j}u)du\\
		&	=2^{-s}\int_0^1e^{2\pi im(x+u)}(\cos2\pi mu+\cos(-2\pi mu))du\\
		&=2^{1-s}\int_0^1e^{2\pi im(x+u)}\cos(2\pi mu)du\\
		&=2^{1-s}e^{2\pi imx}\int_0^1\cos^2(2\pi mu)du\\
		&\qquad \qquad +i\cdot 2^{1-s}e^{2\pi imx}\int_0^1\sin(2\pi mu)\cos(2\pi mu)du\\
		&=2^{-s}e^{2\pi imx},
	\end{align}
completing the proof.
\end{proof}
\begin{proof}[Proof of \trm{T2}]
By (P2) each $m\in \zN(l)$ has a unique representation \e{d31} and so $m$ can be uniquely determined by a set $A=\{{j_1},\ldots,{j_l}\}\subset \ZN$, and a sequence  $\varepsilon=\{\varepsilon_j=\pm 1;\, j=1,2,\ldots, l\}$. Thus we can consider $m=m(A,\varepsilon)$ as a function on $A$ and $\varepsilon$. Thus any finite sum may be written in the form
	\begin{equation*}
		S(x)=\sum_{m\in \zN(l)}c_me^{2\pi imx}=\sum_{A}\sum_{\varepsilon}c_me^{2\pi imx}.
	\end{equation*}
For $t\in (0,1)$ we consider also the sum
	\begin{equation*}
		S_t(x)=\sum_{A}\left(\sum_{\varepsilon}c_me^{2\pi imx}\right)w_A(t),
	\end{equation*}
where 
\begin{equation*}
	w_A(t)=\prod_{j\in A}r_j(t)
\end{equation*}
denotes the Walsh function corresponding the set $A=\{{j_1},\ldots,{j_l}\}$. Since by H\"older's inequality
	\begin{equation*}
		\sum_{A}\left|\sum_{\varepsilon}c_me^{2\pi imx}\right|^2\le 2^{l}\sum_m|c_m|^2,
	\end{equation*}
from the Rademacher chaos bound \e{d2} we obtain
	\begin{equation*}
		\int_0^1\int_0^1|S_t(x)|^pdtdx\le (p-1)^{pl/2}\cdot \left(2^{l}\sum_m|c_m|^2\right)^{p/2}.
	\end{equation*}
	Thus there exists a $t_0\in (0,1)$ such that
	\begin{equation}\label{d47}
		\int_0^1|S_{t_0}(x)|^pdx\le (p-1)^{pl/2}\cdot \left(2^{l}\sum_m|c_m|^2\right)^{p/2}.
	\end{equation}
	Consider the Riesz product
	\begin{align}
		g(x)&=\prod_{j=1}^n(1+r_j(t_0)\cos 2\pi n_jx)\\
		&=1+\sum_{A\subset \{1,2,\ldots,n\}}\left(\prod_{j\in A}r_j(t_0)\cos 2\pi n_j x\right)\\
		&=1+\sum_{A\subset \{1,2,\ldots,n\}}\left(w_A(t_0)\prod_{j\in A}\cos 2\pi n_j x\right),\label{d51}
	\end{align}
where $n$ is a bigger enough integer. It is clear
\begin{equation}\label{d48}
	g(x)\ge 0,\quad \int_0^1g(x)dx=1,
\end{equation}
where the latter follows from formula \e{d46} and the second part of (P2). Thus, applying \lem{D1}, one can write
\begin{equation*}
	S(x)=2^l\int_0^1S_{t_0}(x+u)g(u)du.
\end{equation*}
Then, applying Jensen's inequality, \e{d47} and \e{d48}, we easily get
	\begin{align}
	\int_0^1\left|\frac{S(x)}{2^l}\right|^pdx&\le \int_0^1\int_0^1|S_{t_0}(x+u)|^pg(u)dudx\\
	&\le  ({p-1})^{pl/2}\cdot \left(2^{l}\sum_m|c_m|^2\right)^{p/2}.
	\end{align}
This implies \e{d50}.
\end{proof}
\section{Proof of \cor{C1}}
As it was stated in the introduction that the first part of \cor{C1} immediately follows from \trm{T1}. So we will only give here a counterexample showing sharpness of $\lambda_l$ both in \trm{T1} and \cor{C1}.  According to classical results stated in the introduction there is a trigonometric series 
\begin{equation*}
\sum_{m=1}^\infty c_me^{imx},\quad \sum_{m=1}^\infty |c_m|^2<\infty,
\end{equation*} 
 which is divergent after some rearrangement of the terms. Thus it is enough to prove the following lemma.
 \begin{lemma}
 	There exists a $\lambda_l$-lacunary sequence $\zN$ such that $\zN(l)$ contains all the integers $m\ge 3^l$.
 \end{lemma}  
\begin{proof}
	For every $m\ge 3^l$ consider the group of integers $\{n_k(m):\, k=1,2,\ldots, l\}$ defined by
 \begin{align}
  &n_k(m)=[10^{ml}\lambda_l^{k}]+3^{k},\quad k=1,2,\ldots, l-1,\\ &n_l(m)=m+n_{l-1}(m)+\ldots+n_1(m).
 \end{align}
Recalling that $\lambda_l$ satisfies equation \e{d28}, a rough calculation show that all the groups together, i.e.
\begin{equation*}
	\zN=\{n_k(m):\, 1\le k\le l,\, m\ge 3^l\}
\end{equation*}
form a $\lambda_l$-lacunary sequence.  Besides, every $m\ge 3^l$ is written in the form 
\begin{equation}
m=n_l(m)-n_{l-1}(m)-\ldots-n_{1}(m)
\end{equation}
that means that $\zN(l)$ contains all the integers $m\ge 3^l$. So the proof is complete.
\end{proof}

\section{Proofs of Theorems \ref{T6} and \ref{T7}}\label{S5}

\begin{lemma}\label{L11}
	Let $l\ge 2$ and $\zN$ be a $\lambda$-lacunary sequence with $\lambda>\lambda_{l+1}$, then for any finite sequence $b=\{b_k:\, k\in  \zN_+^*(l)\}$ and a set $E$ of measure \e{d57} we have
	\begin{equation}\label{d55}
		\int_E\left|\sum_{m\in \zN_+^*(l)} b_me^{2\pi i mx}\right|^2>c\sum_{m\in  \zN_+^*(l)} |b_m|^2
	\end{equation}
\end{lemma}
\begin{proof}
	We have
	\begin{align}
		S=\int_E\left|\sum_{m\in \zN_+^*(l)} b_me^{2\pi imx}\right|^2		&=|E|\cdot \sum_{m\in \zN_+^*(l)} |b_m|^2\\
		&\qquad -\sum_{t,s\in \zN_+^*(l),\, t\neq s}b_{t}b_{s}\int_{E^c}e^{2\pi i(t-s)x}dx\\
		&=S_1-S_2.\label{d52}
	\end{align}
We write $E^c$ in the last integral, since the integrals of $e^{2\pi i(t-s)x}$ over $(0,1)$ are zero. 
Each $m\in \zN_+^*(l)$ has a unique representation
\begin{equation}\label{d59}
	m=n_{j_1}+\ldots+n_{j_s},\quad 1\le s\le l,
\end{equation}
according to property (P3). This defines a one to one mapping $\tau$ assigning a subsets $A=\{j_1,\ldots,j_s\}\subset \ZN$ to $m$ by \e{d59}. If $t=\tau(A)$ and $s=\tau(B)$, then we define
\begin{equation}\label{d83}
	t\wedge s=\tau(A\cap B),\quad t\vee s=\tau(A\cup B),
\end{equation}
where in the last notation we additionally suppose that $A\cap B=\varnothing$ and $\#(A\cup B)\le l$.
 For $m\in \zN_+^*(l)$, denote by $\#(m)$ the number of terms in representation \e{d59}. The second sum in \e{d52} can be written in the form
\begin{equation}
	S_2=\sum_{u,v,z}b_{u\vee z}b_{v\vee z}\int_{E^c}e^{2\pi i(u-v)x}dx,
\end{equation}
where the summation is taken over all combinations of numbers $u,v,z\subset \zN_+^*(l)$, satisfying
\begin{align}
	&\#(z)\le l-1,\label{d82}\\
	&u\wedge v=u\wedge z=v\wedge z=0,\label{d81}\\
	&0\le \#(u),\#(v)\le l-\#(z),\quad u+v>0.\label{12}
\end{align}
Using these notations and H\"older's inequality, we obtain
\begin{align}
	|S_2|&\le \sum_{z:\, \#(z)\le l-1}\left(\sum_{u,v:\, \text { satisfying } \e{d81},\e{12}}|b_{u\vee z}b_{v\vee z}|^2\right)^{1/2}\\
	&\qquad\qquad\qquad\qquad\times \left(\sum_{u,v:\, \text { satisfying } \e{d81},\e{12}}\left|\int_{E^c}e^{2\pi i(u-v)x}dx\right|^2\right)^{1/2}.\label{d77}
\end{align}
From (P4) it follows that each integer $m$ has at most $d(l,\lambda)$ number of representation $u-v$, where $u,v$ satisfy  \e{d81}, \e{12}. So the set of functions $e^{2\pi i(u-v)x}$ are a union of $d(l,\lambda)$ orthonormal systems. Thus, by Bessel inequality we get
\begin{equation}\label{d78}
\sum_{u,v:\, \text { satisfying } \e{d81},\e{12}}\left|\int_{E^c}e^{2\pi i(u-v)x}dx\right|^2\le d(l,\lambda)|E^c|< 2^{-l-1},
\end{equation}
where the last inequality is obtained if we choose $\alpha(l,\lambda)=1-(d(l,\lambda)\cdot 2^{l+1})^{-1}$ in  \e{d57}. On the other hand for a fixed $z$, $\#(z)\le l-1$, we can write
\begin{align}
	\sum_{u,v:\, \text { satisfying } \e{d81},\e{12}}|b_{u\vee z}b_{v\vee z}|^2&\le 	\sum_{u:\, u\wedge z=0,\, \#(u)+\#(z)\le l}|b_{u\vee z}|^2\\
	&\qquad\qquad\times	\sum_{v:\,  v\wedge z=0,\,\#(v)+\#(z)\le l}|b_{v\vee z}|^2\\
	&=\left(\sum_{u:\, u\wedge z=0,\,\#(u)+\#(z)\le l}|b_{u\vee z}|^2\right)^2.\label{d84}
\end{align}
Then, any integer can have at most $2^l$ number of representation $u\vee z$, where $u,z$ satisfy \e{d82}-\e{d81}. Thus from \e{d77} and \e{d78} we obtain
\begin{align*}
	\left|S_2\right|<2^{-l-1} \sum_{z:\, \#(z)\le l-1}\,\,\sum_{u:\, u\wedge z=0\,\#(u)+\#(z)\le l}|b_{u\vee z}|^2\le \frac{1}{2} \sum_{m\in \zN_+^*(l)}|b_m|^2,
\end{align*}
where we use the fact that the number of  representation $s=u\vee z$ of an integer $1\le s\le l$ doesn't exceed $2^l$. Therefore,
	\begin{align}
		S\ge S_1-|S_2|>|E|\cdot \sum_{m\in \zN_+^*(l)} |b_m|^2-\frac{1}{2}\cdot \sum_{m\in \zN_+^*(l)} |b_m|^2>c\sum_{m\in \zN_+^*(l)}|b_m|^2.
	\end{align}
\end{proof}
\begin{proof}[Proof of \trm{T6}]
	Without loss of generality we can suppose  that
	\begin{equation}
		\int_E\left|\sum_{m\in \zN_+^*(l)} t_{n,m}b_me^{2\pi imx}\right|^2<M,\quad n=1,2,\ldots,
	\end{equation}
	as $x\in E$. Then, from \lem{L11} we obtain
	\begin{equation*}
		\sum_{m\in \zN_+^*(l)} |t_{n,m}|^2|b_m|^2\lesssim  \int_E\left|\sum_{m\in \zN_+^*(l)} t_{n,m}b_me^{2\pi imx}\right|^2<M
	\end{equation*}
and, applying \e{d60}, we conclude $\sum_{m\in \zN_+^*(l)} |b_m|^2\le M$.
\end{proof}

\begin{proof}[Proof of \trm{T7}]
	Using the Egorov theorem we can suppose that the sums converge to zero uniformly on $E$ of measure \e{d11}. Thus, applying \lem{L11}, for bigger enough integers $n$ we can write
	\begin{align}
		\varepsilon>\int_E\left|\sum_{m\in \zN_+^*(l)} t_{n,m}b_me^{2\pi imx}\right|^2\ge c	\sum_{m\in \zN_+^*(l)} |t_{n,m}|^2|b_m|^2
	\end{align}
and, once again using \e{d60}, we conclude $b_m=0$, $m\in \zN_+^*(l)$.
\end{proof}
\section{Proofs of Theorems \ref{T4} and \ref{T5}}\label{S6}
The approach used in \sect{S5} may be readily applied in the proofs of Theorems \ref{T6} and \ref{T7}. So we will just briefly state the proof of lemma analogous to \lem{L11}, leaving the proofs of the theorems to the reader.
\begin{lemma}\label{L13}
	Let $l\ge 2$, then for any finite sequence $b=\{b_m:\, m\in  \ZD^*(l)\}$ and a set $E$ of measure \e{d57} we have
	\begin{equation}\label{55}
		\int_E\left|\sum_{m\in \ZD^*(l)} b_mw_m(x)\right|^2>c\sum_{m\in  \ZD^*(l)} |b_m|^2
	\end{equation}
\end{lemma}
\begin{proof}
	We have
	\begin{align}
		S=\int_E\left|\sum_{m\in \ZD^*(l)} b_mw_m(x)\right|^2		&=|E|\cdot \sum_{m\in \ZD^*(l)} |b_m|^2\\
		&\qquad -\sum_{t,s\in \ZD^*(l),\, t\neq s}b_{t}b_{s}\int_{E^c}w_t(x)w_s(x)dx\\
		&=S_1-S_2.\label{52}
	\end{align}
Define a mapping $\tau$, assigning  a set $A=\{k_1,\ldots, k_s\}\subset \ZN$ to the number $m\in \ZD^*(l)$ in \e{d61}.  Using the notations \e{d83}, the second sum in \e{52} can be written in the form
	\begin{equation}
		S_2=\sum_{u,v,z}b_{u\vee z}b_{v\vee z}\int_Ew_u(x)w_v(x)dx,
	\end{equation}
	where the summation is taken over all the integers $u,v,z\in \ZD^*(l)$, satisfying \e{d82}-\e{12}.
	Using H\"older's inequality and writing $w_u\cdot w_v=w_{u+v}$, the second sum in \e{52} may be estimated by
	\begin{align*}
		|S_2|&\le \sum_{z:\, \#(z)\le l-1}\left(\sum_{u,v:\, \text { satisfying }\e{d81},\e{12}}|b_{u\vee z}b_{v\vee z}|^2\right)^{1/2}\\
		&\qquad\qquad\qquad\qquad\times \left(\sum_{u,v:\, \text { satisfying }\e{d81},\e{12}}\left|\int_{E^c}w_{u+v}(x)dx\right|^2\right)^{1/2}
	\end{align*}
Let $d(m)$ denote the number of all possible representation of an integer $m$ by $u+v$, where $u,v\in \ZD^*(l)$, $u\wedge v=0$, $\#(u),\#(v)\le l$. Observe that $d(m)<\frac{2^{2l}}{\sqrt{l}}$ for any $m$. Indeed, such a representation is possible if $\#(m)\le 2l$. One can calculate that the $d(m)$ takes its biggest value when $\#(m)=2l$ and in this case we will have
	\begin{equation*}
		d(m)=\binom{2l}{l}<\frac{2^{2l}}{\sqrt{l}}.
	\end{equation*}
	Then, by Bessel's inequality we get
	\begin{equation}
		\sum_{u,v:\, \text { satisfying }\e{d81},\e{12}}\left|\int_{E^c}w_{u+v}(x)dx\right|^2\le \frac{2^{2l}|E^c|}{\sqrt l}< \frac{1}{2^{2l}\sqrt{l}}.
	\end{equation}
Besides, we can write \e{d84}, and therefore
	\begin{align*}
		\left|S_2\right|< \frac{1}{2^{l}\sqrt[4]{l}} \sum_{z:\, \#(z)\le l-1}\sum_{u:\, u\wedge z=0}|b_{u\vee z}|^2\le \frac{1}{\sqrt[4]{l}}  \sum_{m\in \ZD^*(l)}|b_{m}|^2.
	\end{align*}
	Thus we obtain
	\begin{align}
		S\ge S_1-|S_2|>|E|\cdot \sum_{m\in \ZD^*(l)} |b_m|^2- \frac{1}{\sqrt[4]{l}}  \cdot \sum_{m\in \ZD^*(l)} |b_m|^2>c\sum_{k} |b_k|^2.
	\end{align}
\end{proof}
\section{Proof of \trm{T8}}\label{S10}
We will use the notation $x\oplus y=x+y\mod 1$. 
\begin{lemma}\label{L10}
If $\alpha\in [0,1)$ and an integer $m\in \ZD(l)$ has a representation \e{d70}, then for any $n\in \ZD(l)$ we have
\begin{equation}\label{d65}
	\sum_{\varepsilon_j=0,1}(-1)^{\varepsilon_1+\ldots+\varepsilon_l}w_n\left(\alpha\oplus \frac{\varepsilon_1}{2^{k_1}}\oplus \ldots\oplus \frac{\varepsilon_l}{2^{k_l}}\right)=\left\{\begin{array}{lll}
		0&\text { if }&n\neq m,\\
		\pm 2^l&\text{ if }& n=m.
	\end{array}
	\right.
\end{equation}	
\end{lemma}
\begin{proof}
	If $n\in \ZD(l)$ and $n\neq m$, then there is a term in \e{d70}, say $2^{k_{s}}$ , which is not included in the dyadic representation of $n$.  One can check that
	\begin{equation}\label{d67}
		r_k(x\oplus 2^{-s})=\left\{\begin{array}{rll}
			r_k(x)&\text { if }&k\neq s,\\
			-r_k(x)&\text{ if }& k=s,
		\end{array}
		\right.
	\end{equation}
and so by the definition of a Walsh function we can say that the value of
\begin{equation}\label{d66}
	w_n\left(\alpha\oplus \frac{\varepsilon_1}{2^{k_1}}\oplus \ldots\oplus \frac{\varepsilon_s}{2^{k_s}}\oplus \ldots \oplus \frac{\varepsilon_l}{2^{k_l}}\right)
\end{equation}
doesn't depend on $\varepsilon_s$. 	Thus one can conclude that each term of the sum \e{d65} has its opposite pair and so the whole sum is zero. If $n=m$, then the argument of \e{d67} implies
\begin{equation}
	w_m\left(\alpha\oplus \frac{\varepsilon_1}{2^{k_1}}\oplus \ldots\oplus \frac{\varepsilon_l}{2^{k_l}}\right)=\prod_{j=1}^lr_{k_j}\left(\alpha\oplus \frac{\varepsilon_j}{2^{k_j}}\right)
\end{equation}
and so the sum in \e{d65} is equal
\begin{align*}
		\sum_{\varepsilon_j=0,1}\prod_{j=1}^l(-1)^{\varepsilon_j}r_{k_j}\left(\alpha\oplus \frac{\varepsilon_j}{2^{k_j}}\right)=\prod_{j=1}^l\left(r_{k_j}\left(\alpha\right)-r_{k_j}\left(\alpha\oplus 2^{-k_j}\right)\right)=\pm 2^l.
\end{align*}
Lemma is proved.
\end{proof}
\begin{lemma}\label{L12}
	If $E\subset [0,1)$ has a measure $|E|>1-2^{-l}$, then there exists a $\alpha\in [0,1)$ such that
	\begin{equation}
		\left\{\alpha\oplus \frac{\varepsilon_1}{2^{k_1}}\oplus \ldots\oplus \frac{\varepsilon_l}{2^{k_l}}:\, \varepsilon_j=0,1\right\}\subset E.
	\end{equation}
\end{lemma}
\begin{proof}
	Define the sets $E_j\subset [0,1)$ by the recursive formula
\begin{align*}
E_0=E,\quad E_{j+1}=E_j\cap (2^{-k_j}\oplus E_j),\quad j=0,1,\ldots,l-1.	
\end{align*}
One can check inductively that $|E_{k+1}|>|E_k|-2^{k-l}$. Indeed, we have
\begin{equation*}
	|E_1|=1-|E^c\cup(2^{-k_j}\oplus E_j)^c|>1-2^{1-l},
\end{equation*}
and continuing similarly, we will get $|E_l|>0$. Then on can check that $x\in E_l$ implies 
\begin{equation}
	x\oplus \frac{\varepsilon_1}{2^{k_1}}\oplus \ldots\oplus \frac{\varepsilon_l}{2^{k_l}}\in E
\end{equation}
for any choice of $\varepsilon_j=0,1$. Thus as an $\alpha$ we can arbitrary point of the nonempty set $E_l$.
\end{proof}
\begin{proof}[Proof of \trm{T8}]
Let
	\begin{equation}\label{d68}
		\sum_{k=1}^\infty t_{n,k}a_kw_{m_k}(x)\to 0
	\end{equation}
as $n\to\infty$ on a set $E$, $|E|>1-2^{-l}$. Choose $m\in \ZD(l)$ with a representation \e{d70}. According to \lem{L12}, there exists $\alpha$ such that
	\begin{equation}
		\alpha\oplus \frac{\varepsilon_1}{2^{k_1}}\oplus \ldots\oplus \frac{\varepsilon_l}{2^{k_l}}\in E
	\end{equation}
for any choice of $\varepsilon_j=0,1$. Substituting one of these points in \e{d68} and applying \lem{L12}, we obtain
\begin{equation*}
	t_{n,m}a_mw_m\left(\alpha\oplus \frac{\varepsilon_1}{2^{k_1}}\oplus \ldots\oplus \frac{\varepsilon_l}{2^{k_l}}\right)=0,\quad n=1,2,\ldots.
\end{equation*}
Since $t_{n,m}\to 1$ as $n\to \infty$, applying \lem{L10}, we find
\begin{equation*}
	a_m\sum_{\varepsilon_j=0,1}(-1)^{\varepsilon_1+\ldots+\varepsilon_l}w_m\left(\alpha\oplus \frac{\varepsilon_1}{2^{k_1}}\oplus \ldots\oplus \frac{\varepsilon_l}{2^{k_l}}\right)=
	\pm a_m 2^l=0
\end{equation*}
and so $a_m=0$. Theorem is proved.
\end{proof}

\end{document}